\documentclass[11pt]{article}

\usepackage{amsfonts, amsmath}
\usepackage{amssymb}
\usepackage{amsthm,amscd}
\usepackage{enumerate}
\usepackage{tikz}
\usepackage{lipsum}
\usepackage{subfig}
\usepackage{lineno}
\usepackage{caption}
\usepackage{float}

\newtheorem{theorem}{Theorem}[section]
\newtheorem{corollary}[theorem]{Corollary}

\newtheorem{conjecture}[theorem]{Conjecture}
\numberwithin{equation}{section}
\numberwithin{figure}{section}

\newcommand{\bna}{\begin{eqnarray}}
\newcommand{\ena}{\end{eqnarray}}
\newcommand{\ba}{\begin{eqnarray*}}
\newcommand{\ea}{\end{eqnarray*}}
\newcommand{\bs}[1]{}
\newtheorem{lemma}[theorem]{Lemma}
\newtheorem{remark}{Remark}

\usepackage[margin=1in]{geometry}
\begin{document}
\title{Maximally Dense Disk Packings on the Plane}
\author{Robert Connelly\thanks{Department of Mathematics, Cornell University. \texttt{rc46@cornell.edu} Partially supported by 
NSF grant DMS-1564493.}\, and Maurice Pierre\thanks{Columbia University. \texttt{mp3830@columbia.edu} Partially supported by 
NSF grant DMS-1564493.}}
\date{\today}

\maketitle

\begin{abstract}  Suppose one has a collection of disks of various sizes with disjoint interiors, a packing, in the plane, and suppose the ratio of the smallest radius divided by the largest radius lies between $1$ and $q$. In his 1964 book \textit{Regular Figures} \cite{MR0165423},  L\'aszl\'o Fejes T\'oth  found a series of packings that were his best guesses for the maximum density for any $1> q > 0.2$.  Meanwhile Gerd Blind in \cite{MR0377702, MR0275291} proved that for $1\ge q > 0.74$, the most dense packing possible is $\pi/\sqrt{12}$, which is when all the disks are the same size.  In \cite{MR0165423}, the upper bound of the ratio $q$ such that the density of his packings greater than $\pi/\sqrt{12}$ that  Fejes T\'oth found was $0.6457072159..$.  Here we improve that upper bound to $0.6585340820..$.  Both bounds were obtained by perturbing a packing that has the property that the graph of the packing is a triangulation, which L. Fejes T\'oth called a \emph{compact} packing, and we call a  \emph{triangulated} packing.  Previously all of L. Fejes T\'oth's packings that had a density greater than $\pi/\sqrt{12}$ and $q > 0.35$ were based on perturbations of packings with just two sizes of disks, where the graphs of the packings were triangulations.  Our new packings are based on a triangulated packing that has three distinct sizes of disks, found by Fernique, Hashemi, and Sizova, \cite{MR4292755},  which is something of a surprise.


\end{abstract}

\section{Introduction}
A disk packing is called \emph{compact} or \emph{triangulated} if its contact graph is triangulated, i.e. the graph formed by connecting the centers of every adjacent disk consists only of triangular faces. We are interested in packings on the flat torus, which are equivalent to doubly periodic packings in the plane. The problem is to find and classify all compact packings of order $n$ on the torus, meaning the packing uses $n$ different sized disks. For $n=1$, only a single triangulated packing exists, the hexagonal lattice, where each disk touches six others of the same size. For $n=2$, nine triangulated packings exist, found by Kennedy \cite{MR2195054}. For $n=3$, there are 164 triangulated packings, found by Fernique, Hashemi, and Sizova    \cite{MR4292755}. (This is much less than the upper bound of 11,462 packings given by Messerschmidt \cite{MR4033125}.)

 In 1890, Thue gave the first proof that the hexagonal lattice is the densest single size disk packing. However, some considered his proof to be incomplete. In 1940, L. Fejes T\'oth provided the first rigorous proof. This caused Fejes T\'oth to consider the densest possible packings with multiple disk sizes. The density of any such packing, if triangulated, must be strictly greater than $\pi/\sqrt{12}$, which is the density of the hexagonal lattice. Here, in Theorem \ref{thm:triangle-pack-min} we provide a simple proof.

For any disk packing, let $0<q\leq1$ be the ratio between the radii of the smallest and largest circles used. Florian derived a formula for an upper bound for the density of a packing depending on its value of $q$:

\begin{equation}\label{eqn:Florian}
s(q)=\frac{\pi q^2+2(1-q^2)\sin^{-1}\big(\frac{q}{1+q}\big)}{2q\sqrt{1+2q}}.
\end{equation}

\begin{theorem}[Florian \cite{MR0155235}]\label{thm:Florian}
If $\delta$ is the density of a  packing in the plane with radii between $1$ and $q$, then $\delta \le s(q)$.
\end{theorem}

In Section 2, we will show that if the packing is doubly periodic and triangulated, this inequality becomes strict.

The function $s(q)$ tends to {1} as $q$ tends to {0}, since arbitrarily small disks can fill up any gaps in a packing. As $q$ approaches {1}, the bound decreases monotonically to $\pi/\sqrt{12}$, recovering the hexagonal lattice.

 This formula is mentioned in Fejes T\'oth's 1964 book \textit{Regular Figures}. Additionally, Fejes T\'oth provides guesses for the densest possible packings with radius ratio greater than or equal to a given $q$ (Figure 1.1). Although none of the guesses exactly reach Florian's bound, some of them come quite close, while others are noticeably lower.

\begin{figure}[H]
\centering
\captionsetup{labelsep=colon,margin=1.5cm}
\includegraphics[scale=0.3]{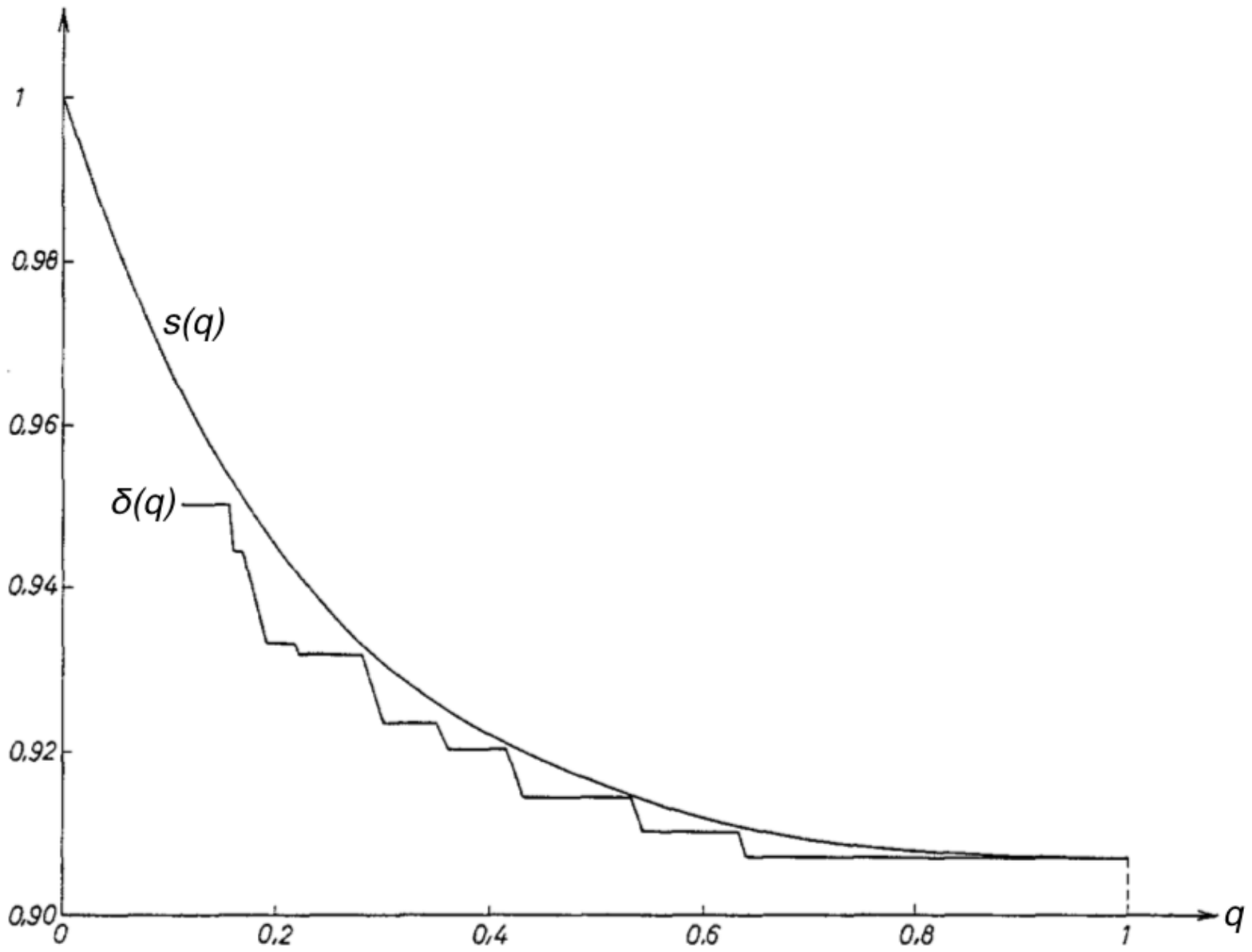}
\caption{Florian's bound $s(q)$ along with $\delta(q)$, the density of Fejes T\'oth's best guess as a function of $q$ reproduced from page 189 of Fejes T\'oth's book \cite{MR0165423}. The final three piecewise sections of $\delta(q)$ will be defined exactly in Section \ref{sect:53}.\quad\cite{MR0377702}}\label{fig:LFT-densities}
\end{figure}

 The problem of finding all triangulated  packings with  $n$ different radii is interesting because we can scour these packings for ones which improve Fejes T\'oth's guesses and come closer to Florian's bound. In fact, with Fernique's three disk packings in \cite{MR4292755} we can do exactly that, which we explain in Section \ref{sect:Fernique}. A knowledge of how close we can get to Florian's bound is important because it helps us with a more general question: Given an arbitrary set of disks with radii between $q$ and 1, what is the densest possible way to arrange them into a periodic packing?

 Another interesting thing to do with these newly-found packings is to find which plane symmetry groups each of them belong to. This is important because the orbifold of a given symmetry group can allow us to construct systematically new packings of any order, although the methods of Kennedy \cite{MR2195054} and Fernique \cite{MR4292755} are sufficient to find all examples for two sizes and three sizes of disks, respectively.


\section{Multiple Size Packings}\label{sect:multiple-size}

If the graph of a packing is a triangulation of the plane, the density of the packing can be calculated by taking an appropriate weighted average of the densities of the packing restricted to each triangle.  Florian's bound (\ref{eqn:Florian}) is the density of a packing of three circular disks in mutual contact, one of radius $r_1$ and two of radius $r_2\le r_1$, where $q=r_2/r_1$, in a triangle formed by their centers.  In general, the density of a packing of $3$ disks of radius $r_i=\tan(\theta_i)$, for $i=1,2,3$ in the triangle formed by the centers is the following function $\delta(\theta_1,\theta_2,\theta_3)$:

\begin{equation}\label{eqn:angle-density}
\delta(\theta_1,\theta_2,\theta_3)=\frac{(\pi/2-\theta_1)\tan^2(\theta_1)+(\pi/2-\theta_2)\tan^2(\theta_2)+(\pi/2-\theta_3)\tan^2(\theta_3)}{\tan(\theta_1)+\tan(\theta_2)+\tan(\theta_3)}
\end{equation}
\begin{figure}[H]
\centering
\captionsetup{labelsep=colon,margin=1.9cm}
\includegraphics[scale=0.2]{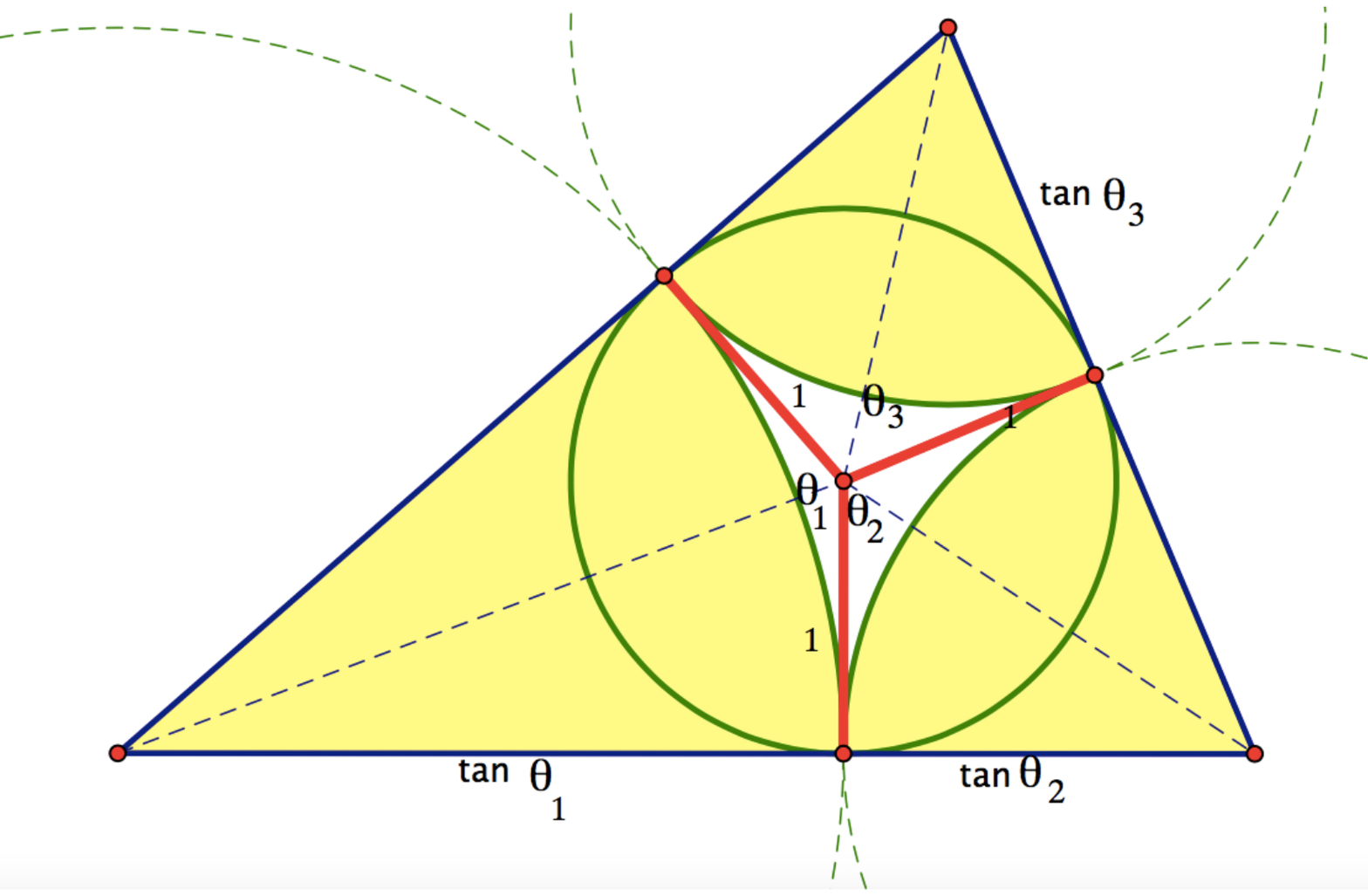}
\caption{Diagram of three mutually tangent circles (dashed green) and the triangle with vertices at their centers (solid blue). There are also the triangle's angle bisectors (dashed blue) and its inscribed circle (solid green).} \label{fig:density}
\end{figure}

The diagram is normalized so that the radius of the incircle is 1. The red segments are radii of the incircle which are orthogonal to the sides of the triangle. $\theta_1$, $\theta_2$, and $\theta_3$ are each angles between one of the angle bisectors and an adjacent red segment and each $\tan \theta_i = r_i$, the $i$-th radius. The density of the packing in the triangle, $\delta(\theta_1,\theta_2,\theta_3)$, is the ratio of the yellow area to the area of the entire triangle. We will show that $\delta$ is minimized when the radii of the three dashed circles are equal.

One can check that when $q=r_2/r_1=r_3/r_1 = \tan (\theta_2)/\tan (\theta_1)= \tan (\theta_3)/\tan (\theta_1)$, then $s(q)=\delta(\theta_1,\theta_2,\theta_3)$. In this case $\theta_1=\pi-2\theta_2=\sin^{-1} (\frac{q}{1+q})$ and $\tan(\theta_1)=\sqrt{\frac{1+2q}{q}}$, $\tan(\theta_2)=\sqrt{1+2q}$.

\subsection{Triangulated Packing's Minimum Density}\label{subsect:triangle-density}

Here we show that the \emph{minimum} density of all triangulated packings is when all the radii of all the disks are equal.

\bigskip
Let the area of the union of the yellow sectors be $A(\theta_1, \theta_2,\theta_2)$ as in Figure {\ref{fig:density}}. Let $a(\theta)=(\pi/2-\theta)\tan^2(\theta)$, then

 \begin{eqnarray} \label{eqn:sector-angle}
 A(\theta_1, \theta_2,\theta_3)&=& a(\theta_1)+ a(\theta_2)+ a(\theta_3)\nonumber\\
 &=&(\pi/2-\theta_1)\tan^2(\theta_1)+(\pi/2-\theta_2)\tan^2(\theta_2)+(\pi/2-\theta_3)\tan^2(\theta_3).
\end{eqnarray}

Twice the area of a right triangle of side length $1$, and angle $\theta$ adjacent to that unit length is $\tan(\theta)$.  Let 
$T(\theta_1, \theta_2,\theta_2)$ be the area of the triangle as in  Figure {\ref{fig:density}}.  Its area is the sum of the areas of the six smaller right triangles, so

 \begin{eqnarray} \label{eqn:triangle-area}
T(\theta_1, \theta_2,\theta_3)=\tan(\theta_1)+\tan(\theta_2)+\tan(\theta_3).
\end{eqnarray}

Thus overall the density of the covered portion of the triangle as in Figure \ref{fig:density} is

\begin{eqnarray}\label{eqn:density-angle}
\delta(\theta_1, \theta_2,\theta_3)&=&A(\theta_1, \theta_2,\theta_3)/T(\theta_1, \theta_2,\theta_3)\nonumber\\
&=&\frac{(\pi/2-\theta_1)\tan^2(\theta_1)+(\pi/2-\theta_2)\tan^2(\theta_2)+(\pi/2-\theta_3)\tan^2(\theta_3)}{\tan(\theta_1)+\tan(\theta_2)+\tan(\theta_3)}
\end{eqnarray}

Here we assume that $\theta_1+\theta_2+\theta_3=\pi$ and each $0<\theta_i<\pi/2$ so that the angles come from the situation of  Figure {\ref{fig:density}}.

We are mainly interested in the following:

\begin{theorem}\label{thm:min-density}  The minimum value of $\delta$ is $\pi/\sqrt{12}$, and is achieved only when $\theta_1=\theta_2=\theta_3=\pi/3$.
\end{theorem}
In other words, this is achieved only when the radii of the touching circles, the $\tan(\theta_i)$ for our normalization are equal.  In order to simplify the calculations, instead of calculating the critical minimum density directly, we will compute the complimentary maximum density 
 
\[
\bar\delta=1-\delta=(T(\theta_1, \theta_2,\theta_3)-A(\theta_1, \theta_2,\theta_3))/T(\theta_1, \theta_2,\theta_3).
\]  

This is the ratio of the curvilinear triangle in the unit circle over the area of the larger triangle that contains the unit circle.  This result follows from the next theorem.

\begin{lemma}Subject to $\theta_1+\theta_2+\theta_3=\pi$ and each $0<\theta_i<\pi/2$ the maximum value of  $T(\theta_1, \theta_2,\theta_3)-A(\theta_1, \theta_2,\theta_3)$ is achieved only when $\theta_1=\theta_2=\theta_3=\pi/3$ and the minimum value of 
$T(\theta_1, \theta_2,\theta_3)$ is achieved only when $\theta_1=\theta_2=\theta_3=\pi/3$.
\end{lemma} 

To do this fix one of the angles, say $\theta_3$, and then regard $\theta_2$ as a function of $\theta_1$, where $\theta_2=\pi-\theta_1-\theta_3$.  Since $\tan(\theta_3)$ and $\theta_3$ are constant, we have the following:

\begin{lemma} The maximum of $\tan(\theta_1)-a(\theta_1)+\tan(\pi-\theta_1-\theta_3)-a(\pi-\theta_1-\theta_3)$ and the minimum of $\tan(\theta_1)+\tan(\pi-\theta_1-\theta_3)$ occur only when $\theta_1=\theta_2=\pi-\theta_1-\theta_3$.
\end{lemma}

\begin{proof}  First the $\tan(\theta)$ case.  For $\theta_3$ fixed, it is clear that $\theta_1=\theta_2=\pi-2\theta_1$ is a critical point.  We calculate the derivatives for $0< \theta<\pi/2$,
\begin{eqnarray*}\label{eqn:area-angle}
 \tan'(\theta)&=&1+\tan^2(\theta)\\
  \tan''(\theta)&=&2\tan(\theta)(1+\tan^2(\theta))>0.
\end{eqnarray*}
Thus $(\theta_1, \theta_1,\pi-2\theta_1)$ is the unique minimum point for $T$ when $\theta_1+\theta_2+\theta_3=\pi$ and each $0<\theta_i<\pi/2$.

For $a(\theta)$ the argument is similar.
\begin{eqnarray*}
 a'(\theta)&=&1+2\tan^2(\theta)-(\pi-2\theta)\tan(\theta)(1+\tan^2(\theta))\\
  a''(\theta)&=&\frac{(2\pi-4\theta)\cos^2(\theta)+6\sin(\theta)\cos(\theta)-3\pi+6\theta}{\cos^4(\theta)}\nonumber
   <0.
\end{eqnarray*}
The last inequality is verified by Maple.  Applying this to each pair of $\theta_i$ at a time, we get that the only overall minimum point for $\delta$ is when $\theta_1=\theta_2=\theta_3=\pi/3$. \qed
\end{proof}

\begin{remark}  The calculation for the tangent above essentially verifies the well-known fact that the maximal area of a triangle enclosing a fixed circle occurs when the triangle is equilateral.
\end{remark}

\begin{corollary}[\ref{thm:min-density}]\label{thm:triangle-pack-min}
The density of a triangulated/compact doubly periodic disk packing, with at least two distinct sizes of disks, is strictly greater than $\pi/\sqrt{12} = 0.9068996821..$.
\end{corollary}

By Theorem \ref{thm:min-density} the density of any packing, restricted to any triangle in that packing is at least $\pi/\sqrt{12}$ and is strictly greater unless all the radii of the triangle are the same.  Since the density of the whole packing is a weighted average of the densities of each triangle, when at least two radii are used, the overall density is strictly greater than $\pi/\sqrt{12}$.


\subsection{Density in Terms of Radii}\label{subsect:radii}

Expression (\ref{eqn:angle-density}) for the density of three disks in a triangle is in terms of the angles of the bounding triangle which useful for Theorem \ref{thm:triangle-pack-min}, but it is also useful to write the same density in terms of the three radii that determine the triangle.  From Heron's formula for a triangle, the area of the triangle is 

\begin{eqnarray*}\label{eqn:Heron}
T_r=T_r(r_1,r_2,r_3)=\sqrt{r_1r_2r_3(r_1+r_2+r_3)}=R(r_1+r_2+r_3),
\end{eqnarray*}
where $R$ is the inradius (that was assumed to be $1$ in Figure \ref{fig:density}).  Thus the inradius is 

\begin{eqnarray*}\label{eqn:inradius}
R=R(r_1,r_2,r_3)=\sqrt{r_1r_2r_3/(r_1+r_2+r_3)}.
\end{eqnarray*}

Thus the density of three disks in a triangle as in Figure (\ref{fig:density}) from Equation (\ref{eqn:density-angle}) is 

\begin{eqnarray}\label{eqn:density-radius}
\delta_r(r_1, r_2,r_3)=\frac{(\pi/2-\tan^{-1}(r_1/R))r_1^2+(\pi/2-\tan^{-1}(r_2/R))r_2^2+(\pi/2-\tan^{-1}(r_3/R))r_3^2}{T_r}
\end{eqnarray}

Then one can check that Florian's bound Equation (\ref{eqn:Florian}) for $0<q \le 1$ is

\begin{eqnarray*}\label{eqn:Florian-derived}
s(q)=\delta_r(1, q,q).
\end{eqnarray*}


\subsection{Comments about the Florian Bound}\label{subsect:comments}

Part of Florian's bound is that if there are two sizes of disks, large and small,  and one puts three disks in contact as in Figure \ref{fig:density}, there are three ways to do it, all the same size, which has density, $\pi/\sqrt{12}$, or large-large-small, or large-small-small.  Theorem \ref{thm:min-density} shows that when all three have the same size, the density is the smallest of the three cases.  The large-small-small case always has the largest density.  Figure \ref{fig:comparison-intermediate} shows a typical case for the density in a triangle of $\delta_r( r_0, r,1)$, where $ 0 \le r_0 \le r \le 1$.

\begin{figure}[H]
\centering
\includegraphics[scale=0.4]{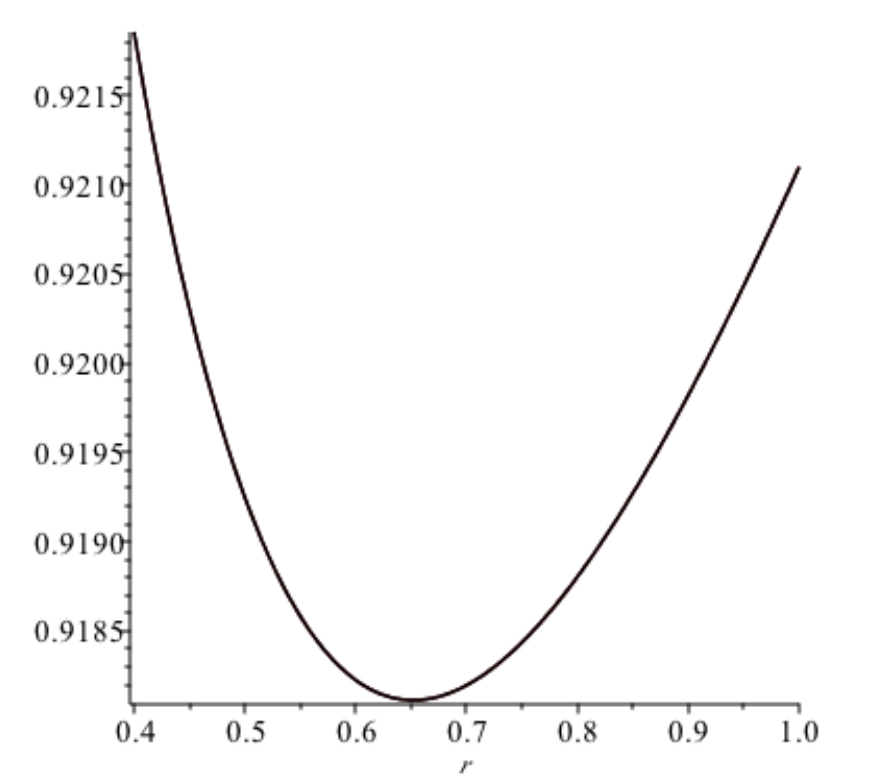}
\captionsetup{labelsep=colon,margin=1.0cm}
\caption{This shows the density $\delta_r(r_0, r, 1)$ of a packing in a triangle, where the  smallest radius $r_0=0.4$  and largest $1$ are fixed, with the intermediate radius $r_0 = 0.4 \le r \le 1$ varying between the two.  Note that the ends of the interval $r=r_0$ and $r=1$ have the largest density locally, with the  large-small-small case $ 1, r_0, r_0$ having the largest density globally.}\label{fig:comparison-intermediate}
\end{figure}

On the other hand, if one compares the density of the two ends of the interval in Figure \ref{fig:comparison-intermediate}, the ratio of the two densities is very close to $1$.

\begin{figure}[H]
\centering
\includegraphics[scale=0.6]{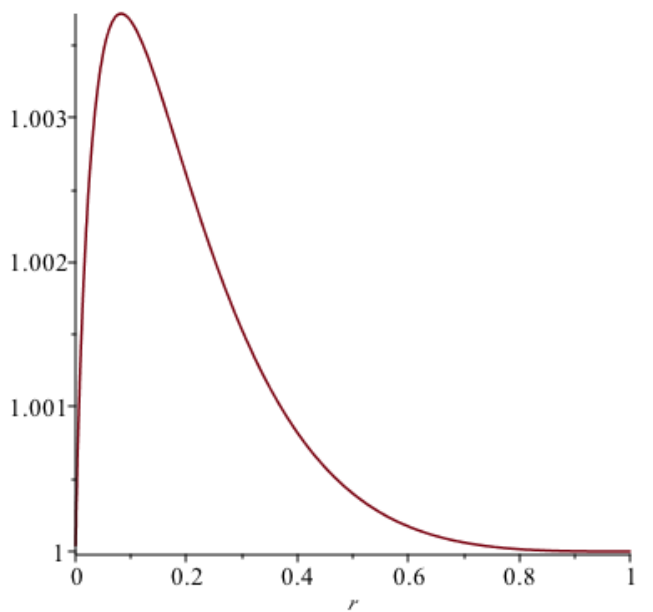}
\captionsetup{labelsep=colon,margin=1.3cm}
\caption{This shows the ratio of $\delta_r(1,r,r)/\delta_r(1,1,r)$, for $0<r<1$, and it appears that for all $r$, the ratio is less than $1.00372119$.  Although, the large-small-small case always has higher density, the density $\delta_r(1,1,r)$ is quite close to $\delta_r(1,r,r)$.}\label{fig:ratios}
\end{figure}


\subsection{The Florian Bound is Never Achieved}\label{subsect:Florian-off}

In \cite{MR2007963}  Alad\'{a}r Heppes said ``\emph{The upper bounds given by L. Fejes T\'oth and Moln\'ar [FM] for the least upper bound $\delta(1,r)$ of the density of a packing of unit disks and disks of radius $r < 1$ have been sharpened by Florian \cite{MR0155235}, who proved that the density cannot exceed the packing density within a triangle determined by the centers of mutually touching circles of radius $1, r$ and $r$. Unfortunately, such packings do not tile the plane for any value of $r$, thus this general bound is never sharp.}"

We explain that last statement here.
We assume that the packing is periodic with a finite number of packing disks per fundamental region, say.  Equivalently, this means that the packing is a collection of circular disks with disjoint interiors in a flat torus, which is determined by some lattice with two independent generators.  For any such packing, normalize the largest radius of a packing to be $1$, and suppose that the smallest radius of the packing is $r_0 < 1$.  Let the other radii of any triangle that contains the radius $1$ be $r_1 \le r_2 \le 1$ and $r_0 \le r_1$ of course. Then 

\begin{eqnarray}\label{eqn:density-monotone}\delta_r(r_1, r_2, 1) \le  \delta_r(r_1, r_1, 1) = s( r_1) \le \delta_r(r_0, r_0, 1)=s(r_0),
\end{eqnarray}
from Figure \ref{fig:comparison-intermediate} and the monotone decreasing property of Florian's bound Figure \ref{eqn:Florian}.  Furthermore if either of the inequalities in (\ref{eqn:density-monotone}) is strict, Florian's overall bound for the packing will never be equality for a doubly periodic packing, say.  We prove the following:

\begin{theorem}\label{thm:strict} If $\delta$ is the density of a doubly periodic triangulated  packing in the plane with radii between $1$ and $q$, then $\delta < s(q)$.
\end{theorem}
\begin{proof}Assume that Florian's bound  in Theorem \ref{thm:Florian} is attained, and we look for a contradiction.  Let $r_0$ be the smallest radius of a disk, and $1$ the largest radius.  Choose any disk of radius $1$.  From the discussion above, each of its adjacent disks must have radius $r_0$ as well.  Similarly, the disks in order around any $r_0$, must be alternately $1, r_0, 1, r_0, \dots$,  for an even number of adjacent disks.  Otherwise, we would have three adjacent disks with radii $r_0, r, 1$, with $r_0 < r\le 1$, where the triangle of centers would have density strictly less than $s(r_0)$ contradicting our assumption.  Continuing this way, we see that all the triangles of the triangulation correspond to packing disks with radii, $r_0, r_0, 1$.  Not only that, but the number of disks of size $1$ will be adjacent to exactly, say $n \ge 3$ other disks of radius $r_0$, and each disk of radius $r_0$ will be adjacent to exactly $2m \ge 4$ other disks, $m$ with radius $1$, $m$ with radius $r_0$.  In particular, there will only be two sizes of disks.  All triangulated packings with just two sizes of disks have been found by Kennedy \cite{MR2195054}, see Figure \ref{fig:Kennedy-9}, and they all have at least one triangle that is either equilateral or that corresponds to the $1,1,r_0$. Alternatively, one can use the following Lemma \ref{lemma:triangulated-sizes} that finds all triangulated packings with two sizes of disks, where each disk has radii of size $a, b, b$, where $a \ne b$ are positive radii, and the large-small-small case never appears. \qed
\end{proof}

\begin{figure}[H]
\centering
\includegraphics[scale=0.56]{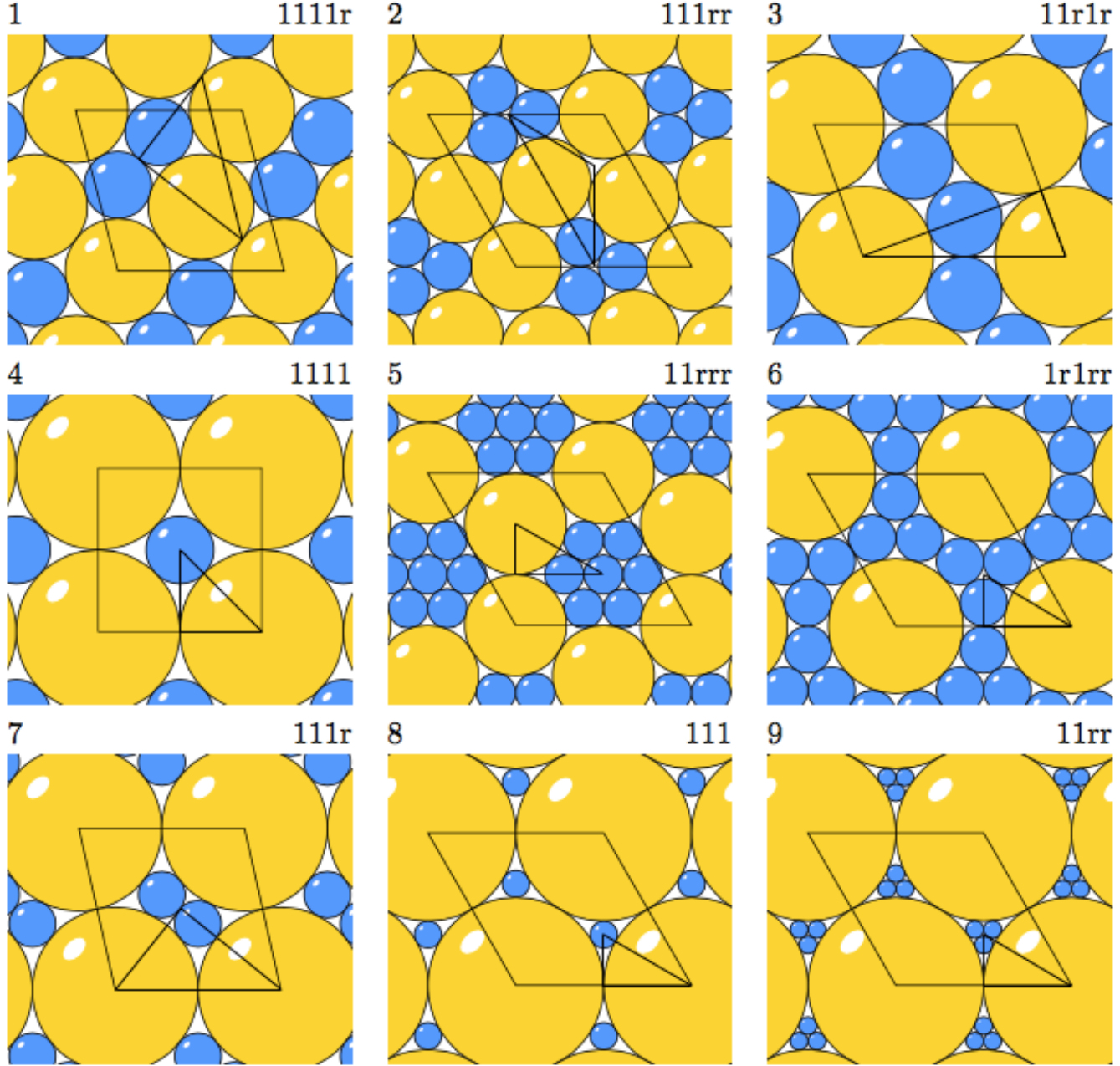}
\captionsetup{labelsep=colon,margin=1.3cm}
\caption{This shows the packings of a torus with two distinct sizes of radii, taken from Kennedy's list in \cite{MR2195054}.}\label{fig:Kennedy-9}
\end{figure}

\begin{lemma}\label{lemma:triangulated-sizes} Suppose we have a triangulated  packing of a flat torus, where the disks corresponding to each triangle have radii $a, b, b$.  Then the only possible packings are Packings $4, 8$ in Figure \ref{fig:Kennedy-9} and the equal radius packing when $a=b$.
\end{lemma}
\begin{proof} Because the shape of each triangle in the triangulation is the same, the number of disks adjacent to a disk with the $a$ radius is the same, say $n \ge 3$, and similarly the number of disks adjacent to one with the $b$ radius is the same even number, say $2m$ for $m \ge 2$, because the neighbors have to alternate as in the proof of Theorem \ref{thm:strict}.  Let $\alpha$ be the half angle at the center of the $a$ radius disk in one of its triangles as in Figure \ref{fig:2-radius}.  Then the angle for the $b$ radius disk in the same triangle is $\pi/2-\alpha$.  So $2\alpha n=2\pi=2m(\pi/2-\alpha)$. Then $\pi/n +\pi/m=\pi/2$, or more simply  $1/n +1/m=1/2$.  Thus the only solutions are $n=3, m=6$, corresponding Packing $8$; $n=4, m=4$ corresponding to Packing $4$; $n$=6, $m$=3 corresponding to when  $a=b$, all in Kennedy's list, Figure \ref{fig:Kennedy-9}, from \cite{MR2195054}.  \qed
\end{proof}

\begin{figure}[H]
\centering
\includegraphics[scale=0.3]{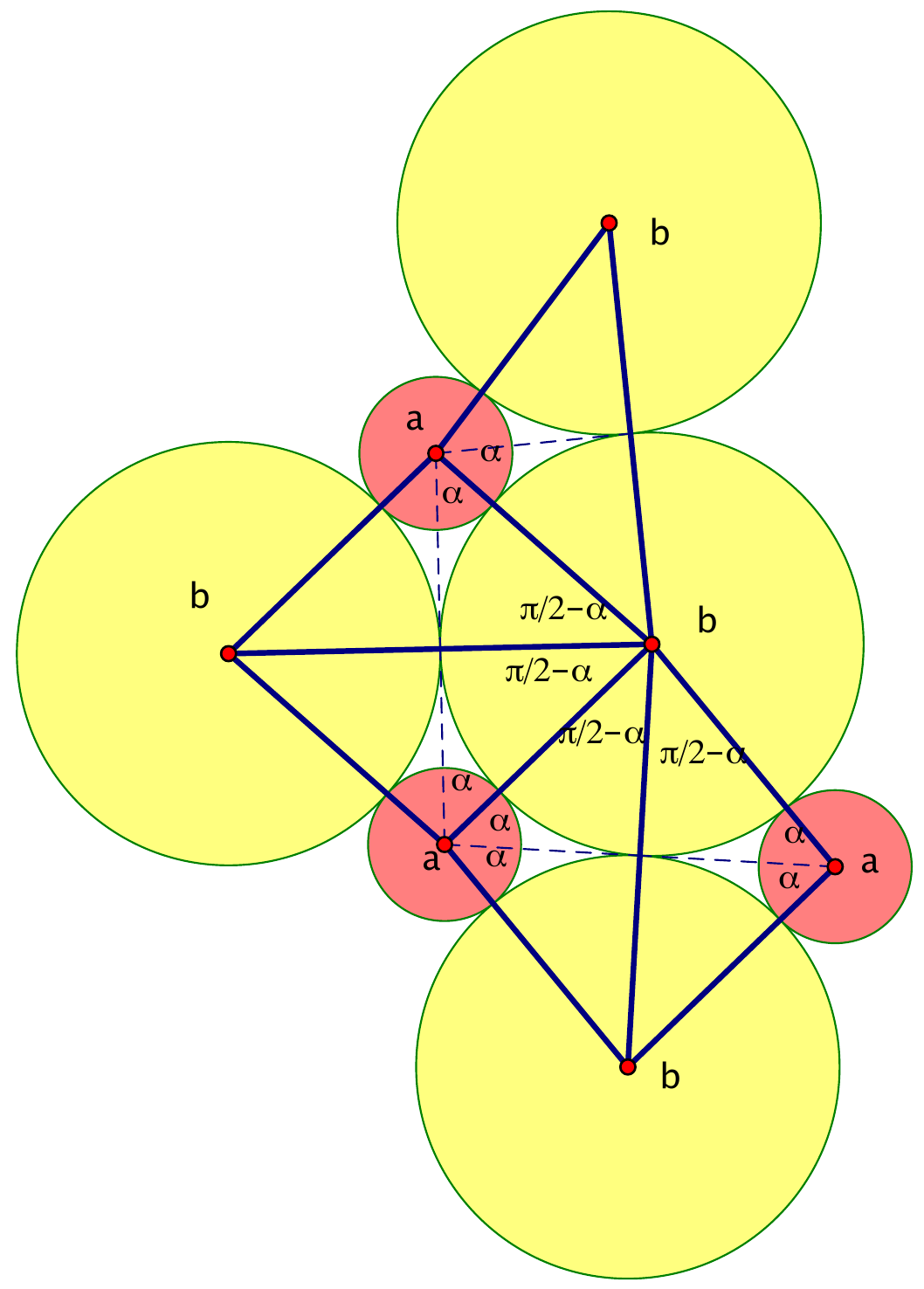}
\captionsetup{labelsep=colon,margin=1.3cm}
\caption{If there is a triangulated packing with only two sizes of disks, this shows how the angles in the triangle must be so the triangulation fits together.}\label{fig:2-radius}
\end{figure}

Notice that the packings of Lemma \ref{lemma:triangulated-sizes} are of the large-large-small type, which do not have the maximum density for a given radius ratio. Nevertheless, they have densities that are very close to Florian's bound in Theorem \ref{thm:Florian}.  Interestingly, Alad\'{a}r Heppes in \cite{MR2007963} proved that for two sizes of disks in the ratio $\sqrt{2}-1$, as in Packing $4$ in Figure \ref{fig:Kennedy-9}, that its density $\pi(2-\sqrt{2})/2=0.9201511858...$ is the maximum possible, while Florian's bound is $0.9208355993..$.

\section{Fejes T\'oth's Packings} \label{sect:53}

Let $q_1= 0.6375559772\ldots$ be the radius ratio of the packing in Fejes T\'oth's book \cite{MR0165423}, (Figure \ref{fig:LFT-figures}  left here), which is the same as Kennedy's first two-disk packing in Figure \ref{fig:Kennedy-9}.

Let $q_2= 0.6457072159\ldots$ be defined such that $\delta_{FT}(q_2)=\pi/\sqrt{12}$, where $\delta_{FT}(q)$ is defined in Equation (\ref{egn:delta}) coming from the middle packing of Figure \ref{fig:LFT-figures}.

 For $q_1<q \leq q_2$, Fejes T\'oth's packing is a version of Kennedy's packing with the radius of the smaller circle increased slightly so that the new ratio is equal to $q$. This, however, causes the packing to no longer be triangulated/compact.

 With some work (See the Appendix), it is possible to write the density of this packing in terms of $q$.  So for $0.6375559772\ldots = q_1 \le q \le q_2= 0.6457072159\ldots$, we have:
 
\begin{equation}\label{egn:delta}
    \delta_{FT}(q)=\frac{\pi(q^2+1)(q+1)^4\sqrt{1+2q}}{4q(2q^2+5q+\sqrt{2q^3+5q^2+2q}+2)(q+\sqrt{2q^3+5q^2+2q})}
\end{equation}
 
 This is the second-to-last piece of $\delta(q)$ shown in Figure \ref{fig:LFT-densities}. So $\delta_{FT}(q_1)=0.9106832003\ldots$ recovers the density of the unperturbed packing, and the density function $\delta_{FT}(q)$ strictly decreases to $\delta_{FT}(q_2)=0.9068996827\ldots=\pi/\sqrt{12}$.  
 This is shown in Figure \ref{fig:graph} with the green line.

 For $q_2<q\leq1$, Fejes T\'oth's guess is simply the hexagonal lattice, with density $\delta_{FT}(q)=\pi/\sqrt{12}$ (Figure \ref{fig:LFT-figures}, right). Note that the packing with ratio $q_2$ is distinct from the hexagonal lattice, despite having the same density.

 \begin{figure}[H]
\centering
\includegraphics[scale=0.5]{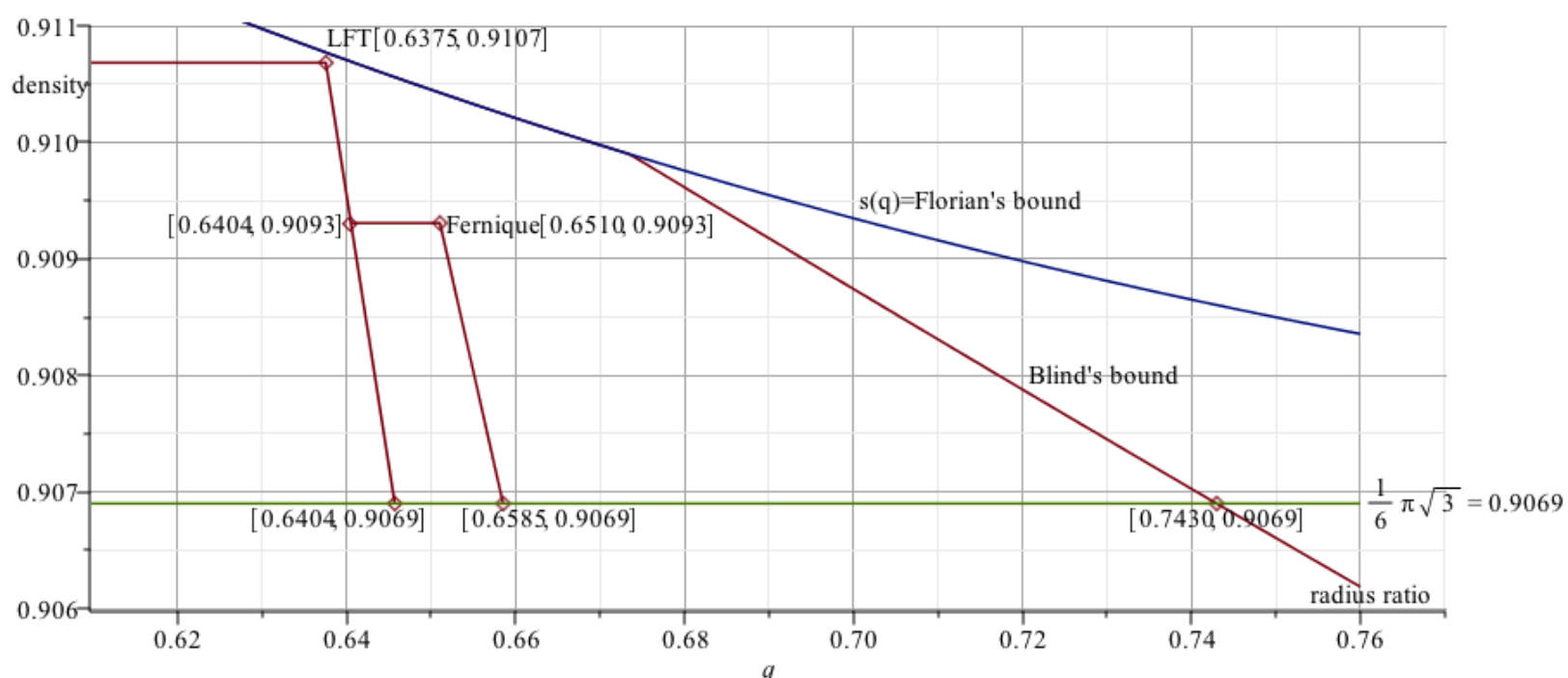}
\captionsetup{labelsep=colon,margin=1.5cm}
\caption{This graph is a magnified version of a portion of the one shown in Figure 1.1, with $s(q)$ in green and $\delta_{FT}(q)=\delta(q)$ in brown. The blue function $\Delta(q)$ shows the improvement to $\delta(q)$. The important values of $q$ are marked by the coordinates of their locations.  See the text for more precise calculations of the values here.}\label{fig:graph}
\end{figure}

 Let 
 
\[q_B=\sqrt{\frac{\sqrt{12}-7\tan(\pi/7)}{5\tan(\pi/5)-\sqrt{12}}}=0.7429909632\ldots\]

\noindent For $q_B\leq q\leq1$, it is known that the hexagonal lattice is the best possible guess. This was shown by Blind in the 1960s in \cite{MR0275291,MR0377702}. He showed that for packing ratio $q$, the density is at most $$\frac{\pi(q^2+1)}{q^2\cdot5\tan(\pi/5)+7\tan(\pi/7)},
$$ 
as is shown in red in Figure \ref{fig:graph}.
\begin{figure}[H]
\centering
\includegraphics[scale=0.15]{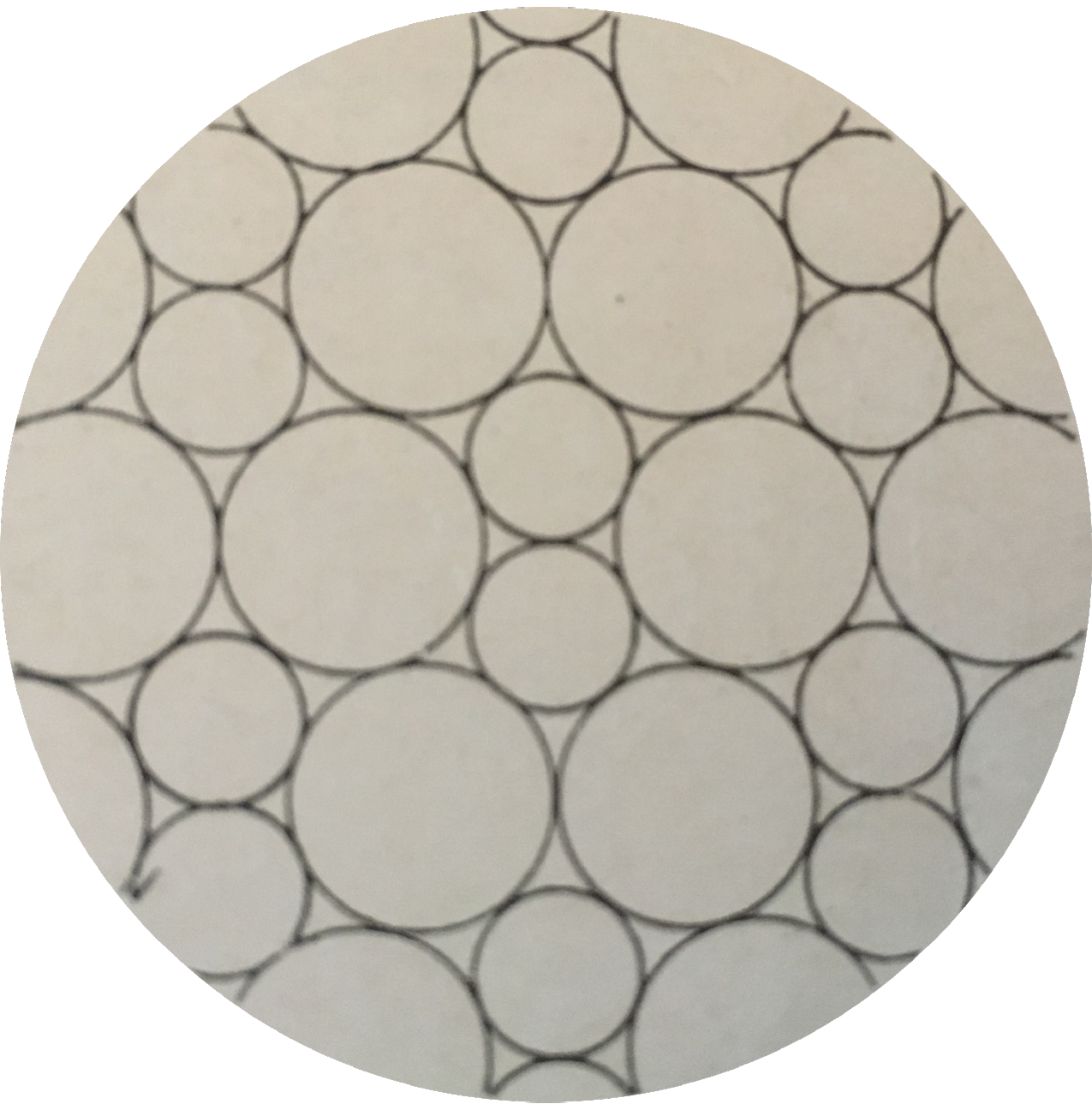}
\includegraphics[scale=0.15]{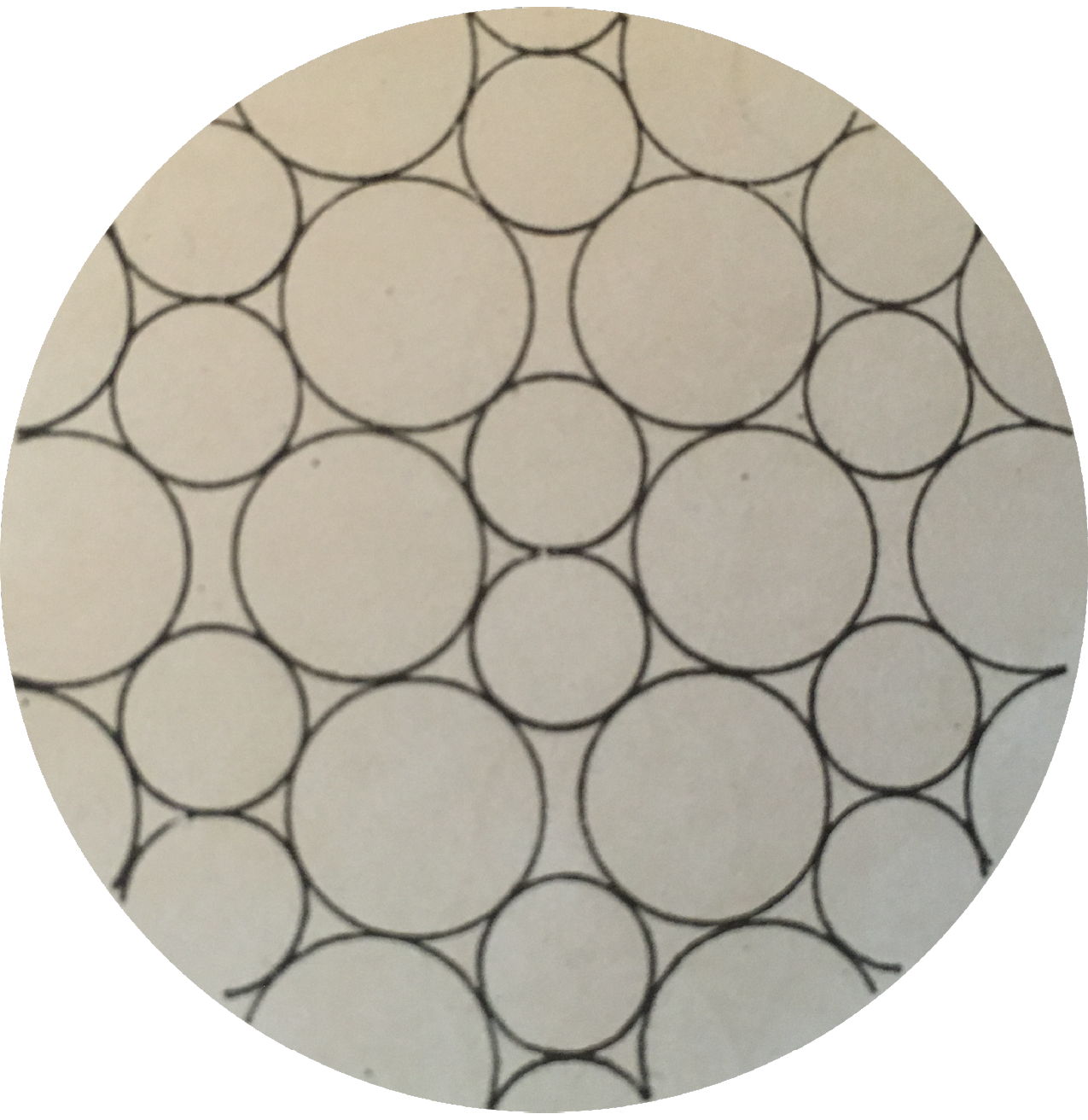}
\includegraphics[scale=0.15]{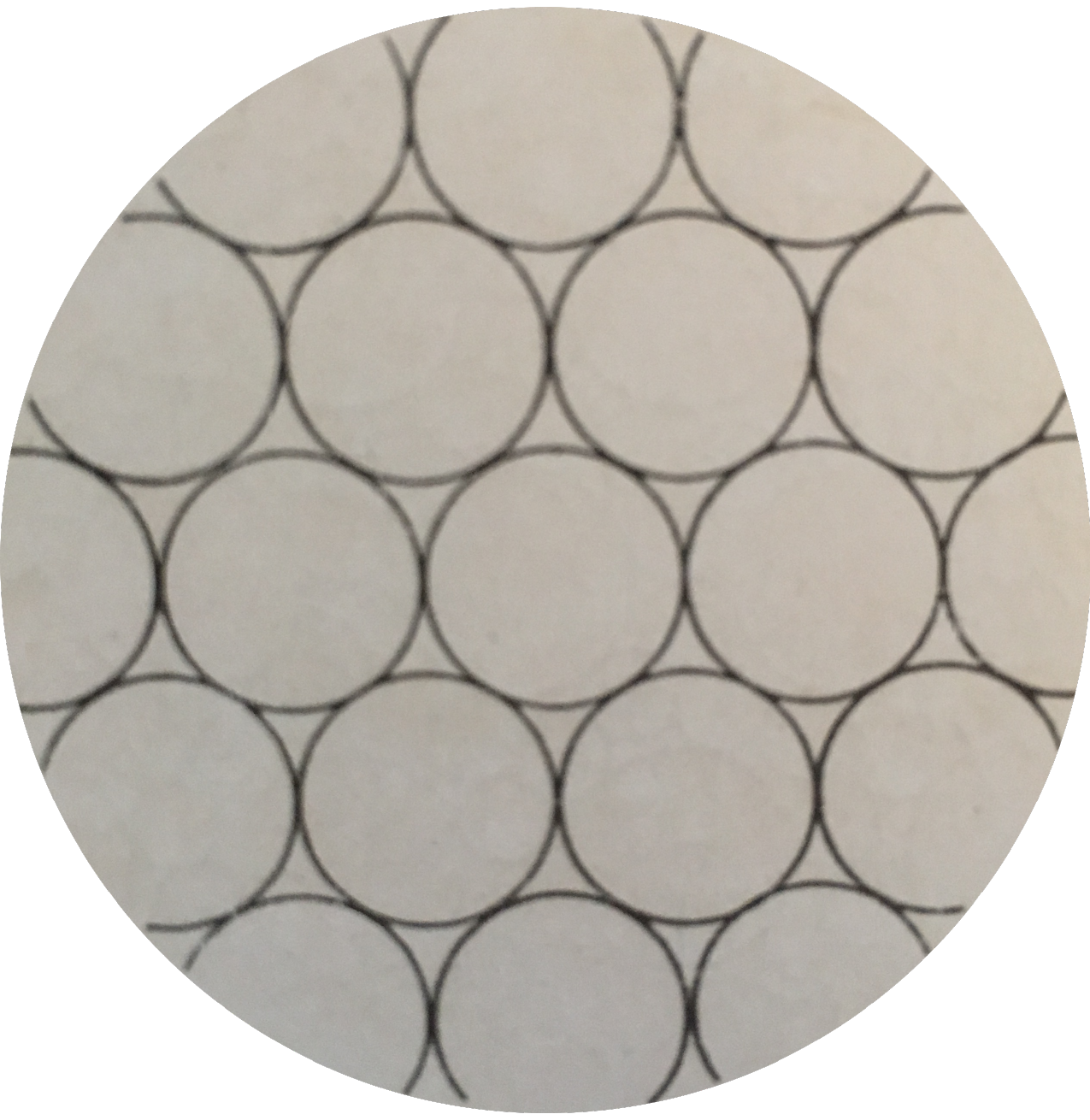}
\captionsetup{labelsep=colon,margin=1.5cm}
\caption{Fejes T\'oth's best guesses, photographed directly from a copy of \textit{Regular Figures}. (Left: $q\leq q_1$, Center: $q_1<q\leq q_2$, Right: $q_2<q\leq1$)\quad\cite{MR0165423}}\label{fig:LFT-figures}
\end{figure}

\section{Fernique's Packings} \label{sect:Fernique}

Packing number 53 in Fernique's list \cite{MR4292755} of compact/triangulated packings of three sizes of disks (which improves a previous estimate in \cite{MR4033125}), Figure \ref{fig:Connelly} left, is of special interest because it has the highest radius ratio of all triangulated/compact disk packings known, aside from the hexagonal lattice. More importantly, Packing 53 improves Fejes T\'oth's packing for the densest packing with radius ratio for $ q_0=0.6404568491\ldots<q \leq 0.6585340820\ldots$ as shown in Figure \ref{fig:graph}, where $q_0$ is such that $\delta_{FT}(q_0)=\delta_{53}$, 
and  $\delta_{53}=0.9093065016\ldots$ is the density of Packing 53 in Fernique's list, (Figure \ref{fig:Connelly}, left).


 Let $q_{53}=0.6510501858\ldots$ be the radius ratio of Packing $53$ (Figure \ref{fig:Connelly}, left).

 Define $\Delta(q)=\delta_{53}$ for $q_0<q\leq q_{53}$, the horizontal red line in Figure \ref{fig:graph}. Since  $\Delta(q)\geq\delta(q)$ for those values of $q$, Packing $53$ is an improvement on Fejes T\'oth's packing for this range.

 For $q_{53}<q<q_B$, Packing 53 is no longer a valid guess because one of the disks in Packing $53$ is smaller than $q$. However, it is possible to create a perturbed version of Packing 53 using Fejes T\'oth's technique in order to make an improved guess for some $q>q_{53}$. We will modify it by increasing the medium radius $p$ and the small radius $q$ according to the following constraint which is satisfied by the unperturbed packing (Appendix \ref{subsection:Fernique-packings}):
 
\begin{equation}\label{eqn:packing-53}
2p^4+(4q+3)p^3+(2q^2-2q+1)p^2-(5q^2+6q)p+q^2=0
\end{equation}

\noindent This is to ensure that the medium sized disks remain in contact with each other (Figure (\ref{fig:Connelly}), right). Using the quartic formula to solve for $p$, we can write the density of the perturbed packing entirely in terms of $q$ (Appendix \ref{subsection:Fernique-packings}):

\begin{equation}\label{eqn:Delta}
\Delta(q)=\frac{\pi(1+p^2+q^2)(p+q)^4(1+p)^2(1+q)^2}{32pq(1+p+q)\big(8p^2q^2-(p^2-6pq+q^2)(1+p+q-pq)\big)\sqrt{pq(1+p+q)}}
\end{equation}

Since the largest disk is normalized to have radius $1$, $q$ is the radius ratio of this packing. $\Delta(q_{53})=\delta_{53}$, and the function $\Delta$ strictly decreases for $q\ge q_{53}$.

\begin{theorem}\label{Thm:ratio-max}
 Let $q_{max}=0.6585340820\ldots$ be defined such that $\Delta(q_{max})=\pi/\sqrt{12}$. This is Fernique's limiting ratio in Figure \ref{fig:graph}.
For $q_{53}<q\leq q_{max}$, $\Delta(q)\geq\delta_{FT}(q)$. Therefore, the perturbed Packing 53 is an improvement on the guess of the hexagonal lattice in this range.
\end{theorem}

These packings and their densities are shown as the red line in Figure \ref{fig:graph}.

\begin{figure}[H]
\centering
\includegraphics[scale=0.31]{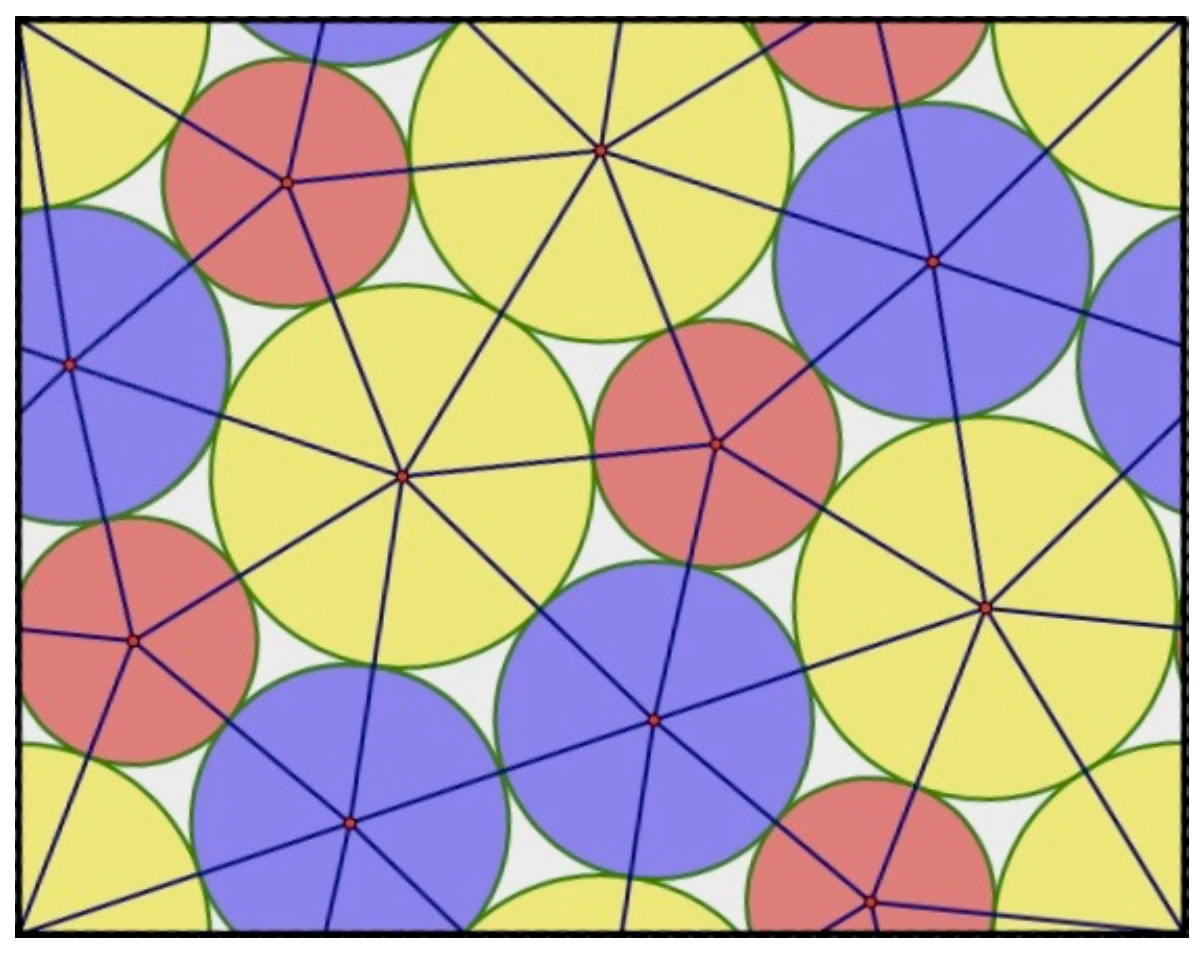}
\includegraphics[scale=0.3025]{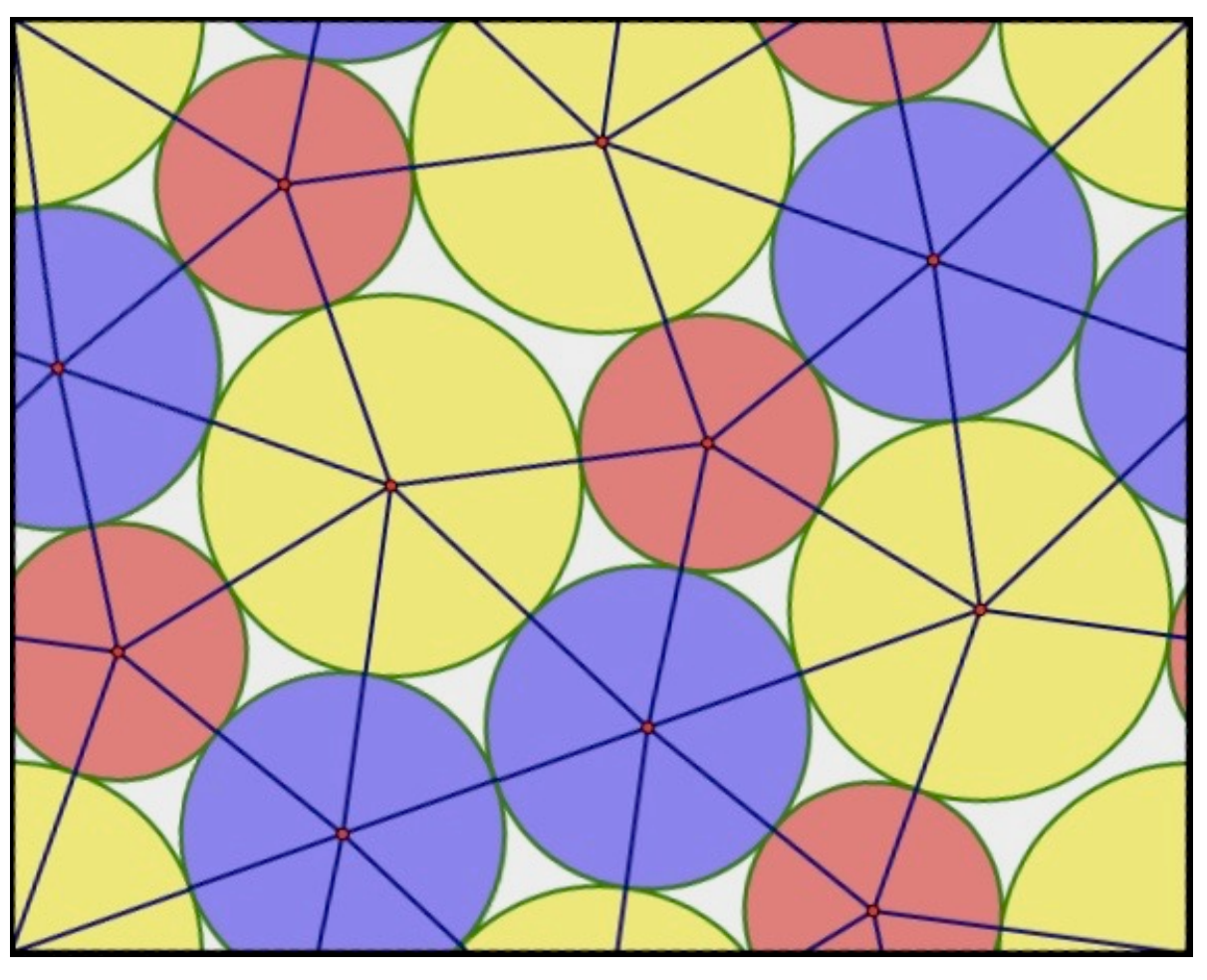}
\captionsetup{labelsep=colon,margin=1.0cm}
\caption{Packing 53 (left) and its perturbed version (right), along with their contact graphs. The rectangular borders of each packing are also their fundamental regions.\quad\label{fig:Connelly}}
\end{figure}

 For $q_{max}<q<q_B$, the perturbed packing is no longer a valid guess because $\Delta(q)$ goes below $\pi/\sqrt{12}$. More packings will need to be discovered and studied, if they exist, in order to make improvements in this range, although the packing in Figure \ref{fig:Connelly}, on the right, may be the limit of those packings that have density at least $\pi/\sqrt{12}$.  Putting this together we propose the following conjectures:
 \begin{conjecture} For any triangulated disk packing in the plane, Fernique's Packing 53 (Figure 4.1, above left) has the largest radius ratio less than $1$.
 
 Furthermore, for any disk packing in the plane with density greater than $\pi/\sqrt{12}$, its radius ratio is less than that of the perturbed Fernique packing (Figure 4.1, above right), which is $q_{max}=0.6585340820\ldots$ from Theorem \ref{Thm:ratio-max}.
 \end{conjecture}
 
 With triangulated packings with three sizes of disks, we know from Fernique's work that his Packing 53 has the largest radius ratio less than $1$, and it seems unlikely that having more sizes of disks allows the radius ratio to be larger. It also seems that triangulated packings are a good first approximation for any packing to have a density greater than $\pi/\sqrt{12}$, and among other deformations of Fernique's packing found, the packing in Figure 4.1 right above is the best.

\section{Appendix}
\subsection{Fejes T\'oth's Packings}
This is a derivation of Formula \ref{egn:delta} for L. Fejes T\'oth's middle packing in Figure \ref{fig:LFT-figures}.  Figure \ref{fig:LFT-displaced} shows a displaced version of L. Fejes T\'oth's packing, concentrating on the upper right trapezoid, which is a fourth of the whole fundamental region and which reflects into the whole fundamental region.  The coordinates of the vertices of the trapezoid are shown.

\begin{figure}[H]
\centering
\captionsetup{labelsep=colon,margin=1.5cm}
\includegraphics[scale=0.6]{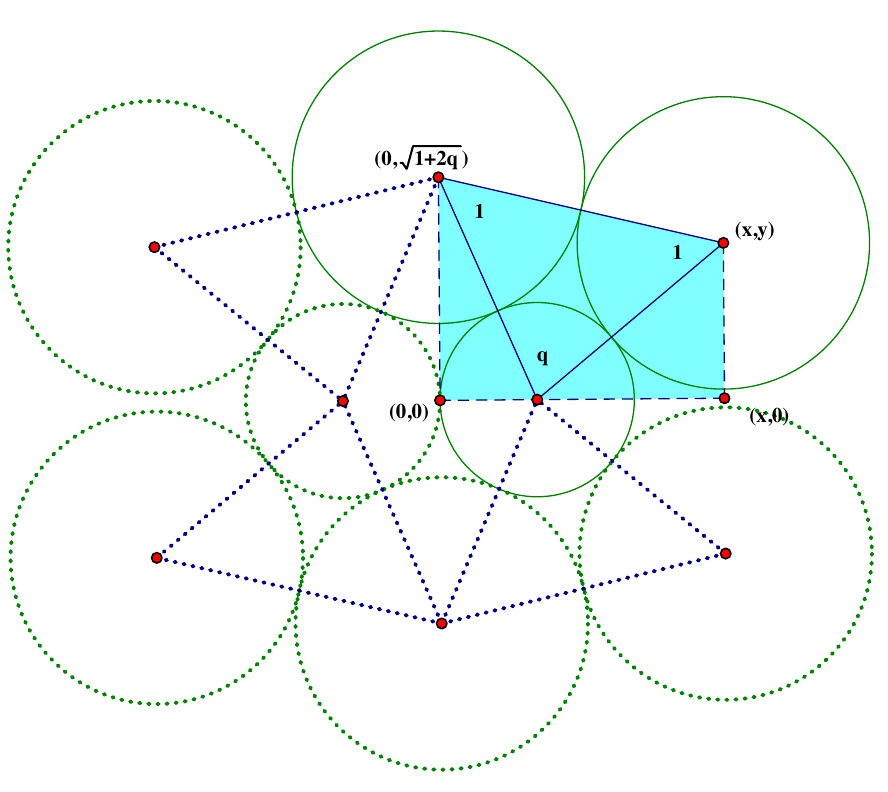}
\caption{Diagram of the fundamental region of L. Fejes T\'oth's packing, with the corner trapezoid, constituting one-fourth of the fundamental region.} \label{fig:LFT-displaced}
\end{figure}

There are two circles of radius $1$, and two circles of radius $q$ per fundamental region in this Figure \ref{fig:LFT-displaced}.
The following equations show the edge length constraints:

\begin{eqnarray*}
|(x,y)-(0,\sqrt{1+2q})|&=& 2\\
|(x,y)-(q,0)|&=& 1+q,
\end{eqnarray*}
which translates to $x^2 +y^2=2y\sqrt{1+2q}-(1+2q)+4=2xq+1+2q$.

\begin{eqnarray*}
x^2+(y-\sqrt{1+2q})^2&=& 4\\
(x-q)^2+y^2&=& (1+q)^2.
    \end{eqnarray*}

So we get:    
\begin{eqnarray*}
x^2 +y^2=2y\sqrt{1+2q}-(1+2q)+4=2xq+1+2q.
\end{eqnarray*}

Solving these two equations for $x$ and $y$ in terms of $q$ we get:
\begin{eqnarray*}
x&=&2\cdot\frac{q+\sqrt{2q^3+5q^2+2q}}{(q+1)^2}\\
y&=&\frac{2q^3+7q^2+4q\sqrt{2q^3+5q^2+2q}-1}{\sqrt{1+2q}(q+1)^2}
\end{eqnarray*}

Then the density of this packing in terms of $q$ is 

\begin{eqnarray*}
\delta=\delta(q)=\frac{\pi(q^2+1)}{(\sqrt{1+2q}+y(q))x(q)}.
\end{eqnarray*}

Note that $y(0.6375559772\ldots)=1$ which means that the configuration is as in Fejes T\'oth's Figure \ref{fig:LFT-figures} on the left, and Kennedy's Figure 1 of Figure \ref{fig:Kennedy-9}, the triangulated packing. In Section \ref{sect:53} $q_1=0.6375559772\ldots$.  Note also that when the ratio $q=1$, $y(1)=\sqrt{3}$, and $x(1)=2$, showing that the configuration is the ordinary hexagonal packing, where  two radius $q$ disks come together and touch.  Note that $\delta(0.6375559772\ldots)=0.9106832003\ldots > \pi/\sqrt{12}$.
When the $q$ disk radius is expanded to $q=0.6457072159\ldots=q_2$, then $\delta(q_2)=\pi/\sqrt{12}$.  So $q_2=0.6457072159\ldots$ is the limit of the largest $q$ radius for Fejes T\'oth's packings.

$q_1=0.6375559772\ldots$ is the root of the following polynomial:
$$x^4-10x^2-8x+9=0\qquad\cite{MR2195054}$$

$q_2=0.6457072159\dots$ is the root of the following polynomial:
$$9x^{15}+81x^{14}+369x^{13}+1161x^{12}+2757x^{11}+4749x^{10}+5805x^9+5445x^8$$
$$+3643x^7+1235x^6+243x^5-1029x^4-969x^3-369x^2-81x-9=0$$

This was calculated by solving $\delta_{FT}(q)=\pi/\sqrt{12}$ in Wolfram Alpha.

\subsection{Fernique's Packings} \label{subsection:Fernique-packings}
 
 For the evaluation of $q_{53}$ and $q_{max}$ we have the following:
 
 $q_{53}=0.6510501858\ldots$ is the root of the following polynomial:
\[
89x^8+1344x^7+4008x^6-464x^5-2410x^4+176x^3+296x^2-96x+1=0\qquad\cite{MR4292755}
\]
$q_{max}=0.6585340820\ldots$ is the root of the following polynomial:
\begin{multline*}
82944x^{31}+2073600x^{30}+25449984x^{29}+204553728x^{28}+1214611776x^{27}+5674077504x^{26}\\
+21595717440x^{25}+68441069376x^{24}+183725780496x^{23}+423619513104x^{22}+846900183408x^{21}\\
+1474917242352x^{20}+2239664278028x^{19}+2959314640332x^{18}+3384242724844x^{17}\\
+3313803241196x^{16}+2719452571159x^{15}+1783910866439x^{14}+815514300847x^{13}+88889109343x^{12}\\
-279883089565x^{11}-346836129933x^{10}-256274678853x^9-138435598005x^8-57157331979x^7\\
-18283967739x^6-4571655651x^5-892845459x^4-132201675x^3-14152347x^2-985635x-33075=0
\end{multline*}
This was calculated by solving $\Delta(q)=\pi/\sqrt{12}$ in Wolfram Alpha.
 \bigskip
 
 For the calculation of the density $\Delta(q)$ of the perturbed (and unperturbed) Fernique packing $53$, we consider the following portion of the packing as follows:
 
\begin{figure}[H]
\centering
\captionsetup{labelsep=colon,margin=1.5cm}
\includegraphics[scale=0.6]{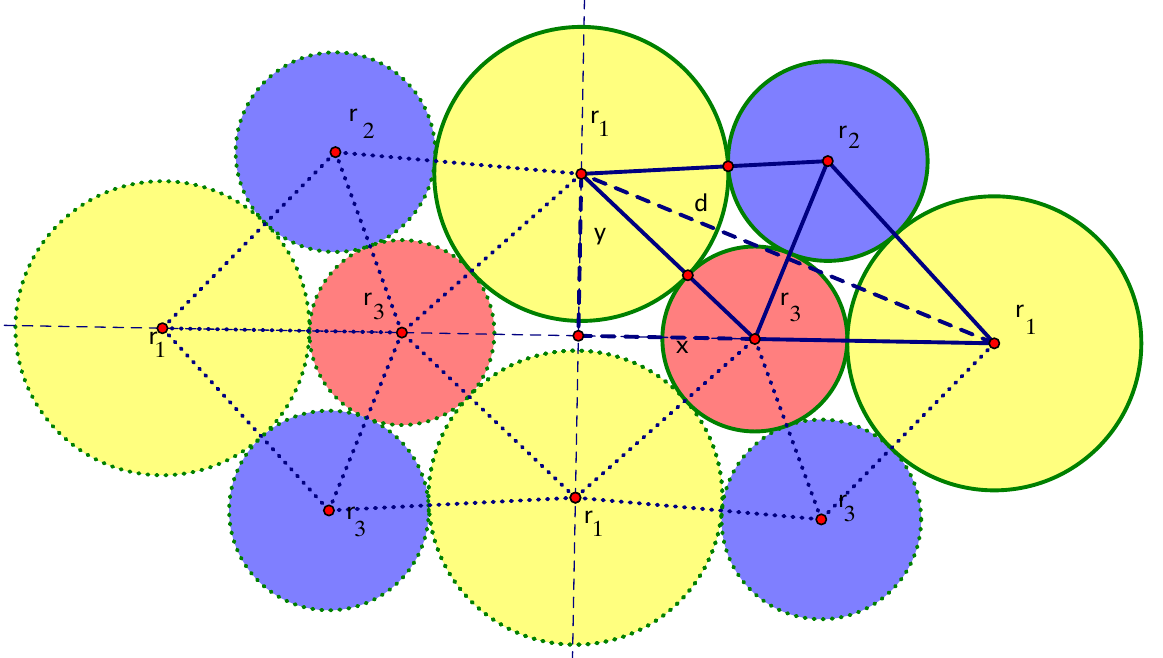}
\caption{Diagram of a portion of the perturbed Fernique packing.  The center point is surrounded symmetrically by four circles with radii, $r_1, r_3, r_1, r_3$ in order with the two $r_1$ disks moved apart slightly.} \label{fig:Fernique-local}
\end{figure}

Fernique's Packing 53 has $180^{\circ}$ rotational symmetry about every  contact point between pairs of large (yellow here) disks and pairs of medium (blue here) disks. Preserving that symmetry yields the periodic packing in Figure \ref{fig:Connelly}, right. In Figure \ref{fig:Fernique-local}, the packing is rotated so that it is easier to see the $180^{\circ}$ rotational symmetry about the midpoint between the yellow disk centers.

Figure \ref{fig:Fernique-pgg} shows how the $10$ disks of Figure \ref{fig:Fernique-local} fit within the fundamental region of Fernique's Packing $53$. The rectangular fundamental region is tiled by $4$ smaller rectangles, each of which is symmetric to its neighbors by a glide reflection either horizontally or vertically. Each glide reflection is a translation along one of the dashed lines followed by a reflection over the same dashed line.
\begin{figure}[H]
\centering
\captionsetup{labelsep=colon,margin=1.5cm}
\includegraphics[scale=0.6]{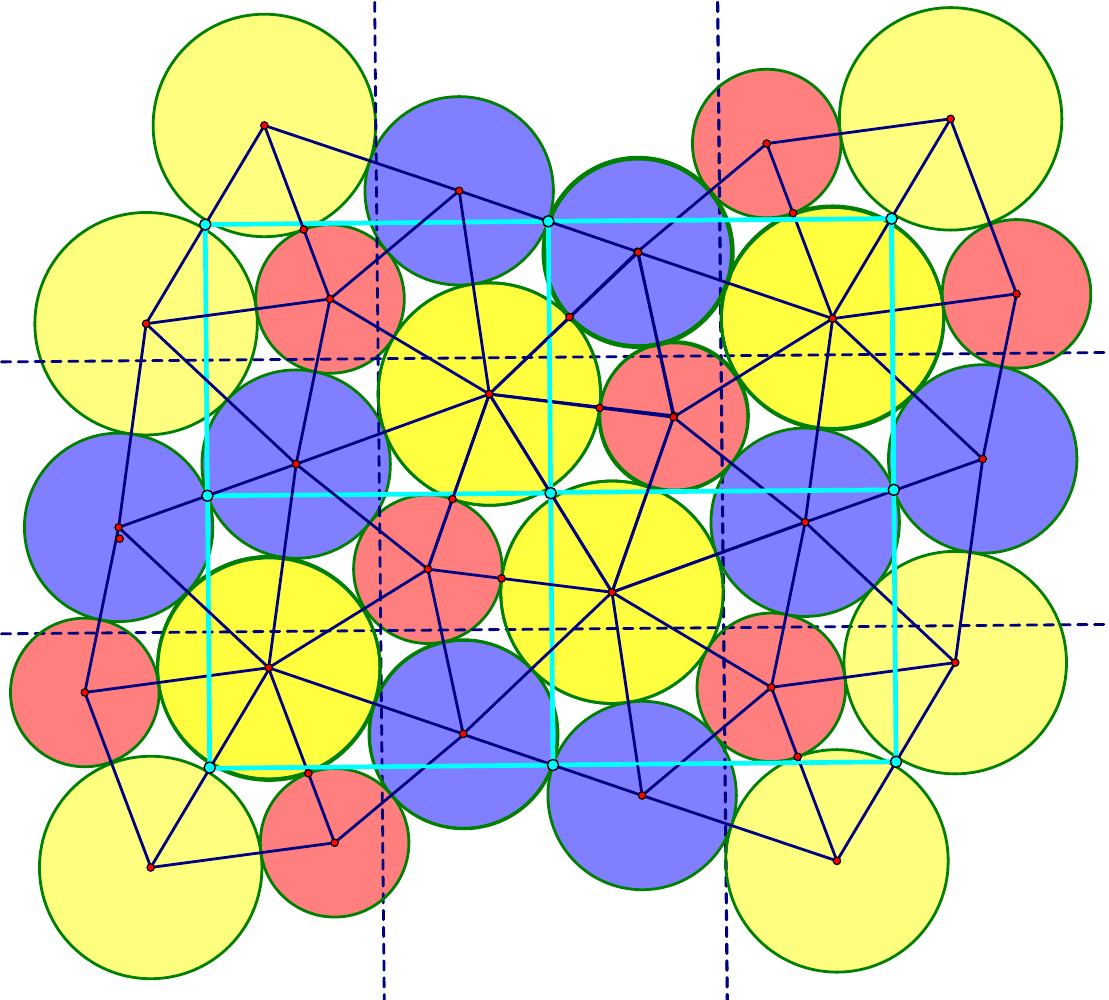}
\caption{Diagram of a fundamental region of the perturbed Fernique packing. The fundamental region is divided into $4$ light blue rectangles. This packing has pgg symmetry, one of the $17$ wallpaper symmetry groups. See \cite{MR992195} for a description of the wallpaper groups.} \label{fig:Fernique-pgg}
\end{figure}

The distance from the $r_1$ disks to the  center of symmetry is defined to be $y$, and the distance from the center of the $r_3$ disk to the center of symmetry is $x$.  The distance $d$ is the distance between the centers of the $r_1$ disks and is perpendicular to the line through the $r_2, r_3$ disk centers as shown.
As before the area of the triangle formed by the centers of the $r_1, r_2, r_3$ disks is
\begin{eqnarray*}
A_{\Delta}=\sqrt{r_1 r_2 r_3(r_1+r_2+r_3)}.
\end{eqnarray*}
The distance $d$ can be calculated because of the symmetry about the line through the centers of the $r_2$ and $r_3$ disks.

\begin{eqnarray*}
\frac{d}{2}(r_2+r_3)\frac{1}{2}=A_{\Delta}.
\end{eqnarray*}
So $d$ can be calculated in terms of the radii.

\begin{eqnarray*}
d^2=\frac{16r_1 r_2 r_3(r_1+r_2+r_3)}{(r_2+r_3)^2}.
\end{eqnarray*}
Then using the right triangle formed by the two $r_1$ circle centers and the center of symmetry,
\begin{eqnarray}\label{eqn:d-distance}
(x+r_1+r_3)^2+y^2=d^2=\frac{16r_1 r_2 r_3(r_1+r_2+r_3)}{(r_2+r_3)^2}.
\end{eqnarray}
Similarly using the using the right triangle formed by the $r_1, r_3$ circle centers and the center of symmetry,
\begin{eqnarray}\label{eqn:x-y-relation}
x^2+y^2=(r_1+r_3)^2.
\end{eqnarray}
Substituting this into Equation (\ref{eqn:d-distance}) we get:
\begin{eqnarray*}
2(r_1+r_3)^2 +2x(r_1+r_3)=\frac{16r_1 r_2 r_3(r_1+r_2+r_3)}{(r_2+r_3)^2}.
\end{eqnarray*}
Solving for $x$ as the following explicit rational function of the radii, we get:
\begin{eqnarray}\label{eqn:x-function}
x(r_1,r_2,r_3)=\frac{8r_1 r_2 r_3(r_1+r_2+r_3)}{(r_2+r_3)^2(r_1+r_3)}-(r_1+r_3).
\end{eqnarray}
Using Equation (\ref{eqn:x-y-relation}) we find $y$ as a function of the radii:
\begin{eqnarray}\label{eqn:y-function}
y(r_1,r_2,r_3)=\sqrt{(r_1+r_3)^2-x(r_1,r_2,r_3)^2}.
\end{eqnarray}

Note that to get a packing (i.e. without overlap) we must have $y\ge r_1$ and $x\ge r_3$, and we can switch the roles of $r_2$ and $r_3$ with the disks switching in Figure \ref{fig:Fernique-local}.  Indeed as the yellow disks move apart and the red disks move together and touch, the packing deforms to the standard hexagonal packing.

Assume that $r_1=1, r_2=p, r_3=q$, and $q<p<1$.  From Figure (\ref{fig:Connelly}) we see that each fundamental region has $4$ disks of each size, so the total area covered by the disks is:

\begin{equation*}
    A(p,q)=4\pi(1+p^2+q^2).
\end{equation*}

This formula is independent of the relative sizes of the disks.

We know that the area of one of the triangles determined by the centers of $3$ different disks is:

\begin{equation*}
    A_{\Delta}(p,q)=\sqrt{pq(1+p+q)}.
\end{equation*}

The area of a quadrilateral (a rhombus) determined by $2$ yellow disk centers and $2$ red disk centers in Figure \ref{fig:Fernique-local}  is:
\begin{equation*}
QUAD(p,q)=2x(1,p,q)y(1,p,q).
\end{equation*}

And the area of a quadrilateral (a rhombus) determined by $2$ yellow disk centers and $2$ blue disk centers in Figure \ref{fig:Fernique-local}  is:
\begin{equation*}
QUAD(q,p)=2x(1,q,p)y(1,q,p).
\end{equation*}

The equation for the constraint on how the sizes of the blue and red disks are increased $(4.1)$ is determined by equating the above expression for the area of the yellow-blue rhombus to a more trivial calculation of the area which assumes that the blue disks are in contact. As a result, this constraint will ensure that the blue disks are always in contact no matter how the packing is perturbed.
\begin{equation*}
2x(1,q,p)y(1,q,p)=\frac{1}{2}\cdot2p\cdot2\sqrt{2p+1}
\end{equation*}

After simplifying the factors of $2$, we obtain
\begin{equation*}
x(1,q,p)y(1,q,p)=p\sqrt{2p+1}.
\end{equation*}

Then, we calculate the value of $x(1,q,p)$.
\begin{equation*}
x(1,q,p)=\frac{8qp(1+q+p)}{(q+p)^2(1+p)}-(1+p)
\end{equation*}
\begin{equation*}
x(1,q,p)=\frac{8qp(1+q+p)-(q+p)^2(1+p)^2}{(q+p)^2(1+p)}
\end{equation*}

Next, we use this expression to calculate the value of $y(1,q,p)$.
\begin{equation*}
y(1,q,p)=\sqrt{(1+p)^2-x(1,q,p)^2}
\end{equation*}
\begin{equation*}
y(1,q,p)=\sqrt{(1+p)^2-\bigg(\frac{8qp(1+q+p)}{(q+p)^2(1+p)}-(1+p)\bigg)^2}
\end{equation*}
\begin{equation*}
y(1,q,p)=\sqrt{\frac{16qp(1+q+p)}{(q+p)^2}-\frac{64q^2p^2(1+q+p)^2}{(q+p)^4(1+p)^2}}
\end{equation*}
\begin{equation*}
y(1,q,p)=\sqrt{\frac{16qp(1+q+p)(q+p)^2(1+p)^2-64q^2p^2(1+q+p)^2}{(q+p)^4(1+p)^2}}
\end{equation*}
\begin{equation*}
y(1,q,p)=\sqrt{\frac{16qp(1+q+p)(p^2+qp-q+p)^2}{(q+p)^4(1+p)^2}}
\end{equation*}

Now, we substitute $x$ and $y$ into the area equality.
\begin{equation*}
\frac{8qp(1+q+p)-(q+p)^2(1+p)^2}{(q+p)^2(1+p)}\sqrt{\frac{16qp(1+q+p)(p^2+qp-q+p)^2}{(q+p)^4(1+p)^2}}=p\sqrt{2p+1}
\end{equation*}

After squaring both sides, we obtain
\begin{equation*}
\frac{\big(8qp(1+q+p)-(q+p)^2(1+p)^2\big)^2}{(q+p)^4(1+p)^2}\cdot\frac{16qp(1+q+p)(p^2+qp-q+p)^2}{(q+p)^4(1+p)^2}=p^2(2p+1).
\end{equation*}

Now, multiply both sides of the equation by $(q+p)^8(1+p)^4$.
\begin{equation*}
\big(8qp(1+q+p)-(q+p)^2(1+p)^2\big)^2\cdot16qp(1+q+p)(p^2+qp-q+p)^2=p^2(2p+1)(q+p)^8(1+p)^4
\end{equation*}

Next, bring all of the terms to one side.
\begin{equation*}
p^2(2p+1)(q+p)^8(1+p)^4-16qp(1+q+p)(p^2+qp-q+p)^2\big(8qp(1+q+p)-(q+p)^2(1+p)^2\big)^2=0
\end{equation*}

We can use Wolfram Alpha to expand and factor this polynomial.
\begin{equation*}
p(p^3-6p^2q+p^2-7pq^2-6pq+q^2)(2p^4+4p^3q+3p^3+2p^2q^2-2p^2q+p^2-5pq^2-6pq+q^2)
\end{equation*}
\begin{equation*}
(p^7+4p^6q+2p^6+6p^5q^2-8p^5q+p^5+4p^4q^3-36p^4q^2-44p^4q+p^3q^4-40p^3q^3-58p^3q^2
\end{equation*}
\begin{equation*}
-48p^3q-14p^2q^4+20p^2q^3+16p^2q^2-16p^2q+33pq^4+48pq^3+32pq^2-16q^4-16q^3)=0
\end{equation*}

Since the perturbation is continuous, the constraint must also be satisfied by the radii of the unperturbed packing. This is only the case for the third factor, and so it can be considered in isolation.
\begin{equation*}
2p^4+4p^3q+3p^3+2p^2q^2-2p^2q+p^2-5pq^2-6pq+q^2=0
\end{equation*}

Lastly, we can factor out identical powers of $p$.
\begin{equation*}
2p^4+(4q+3)p^3+(2q^2-2q+1)p^2-(5q^2+6q)p+q^2=0
\end{equation*}

From Figure (\ref{fig:Connelly}), left that there are $4$ triangles determined by $2$ yellow disks and a red disk corresponding to $2$ QUAD regions.  There are  $4$ triangles determined by $2$ yellow disks and $2$ blue disks corresponding to $2$ other QUAD regions.  Then there are $16$ triangles determined by $3$ different disks. Putting all these regions together we see that the total area of the torus is:
\begin{equation*}
A_T(p,q)=16A_{\Delta}(p,q)+2QUAD(p,q)+2QUAD(q,p).
\end{equation*}

Altogether, we get that the overall density of the packing is:
\begin{equation*}
\Delta(p,q)=\frac{A(p,q)}{A_T(p,q)}.
\end{equation*}

With the constraint equation (4.1), this yields the $\Delta(q)$ given in $(4.2)$.

\newpage
\bibliographystyle{plain}
\bibliography{references}
\end{document}